\documentclass[11pt]{amsart}
\usepackage{geometry}                
\usepackage{graphicx}
\usepackage{amssymb}
\usepackage{epstopdf}
\usepackage{hyperref}
\usepackage{tikz}
\usepackage{pgfplots}
\usepackage{mathrsfs}
\usetikzlibrary{arrows}
\usepackage[all]{xy}
\usepackage{tabularx}
\usepackage{subcaption}
\title[Equilateral dimension of $\fH^n$]{Equilateral dimension of the Heisenberg group $\fH^n$}
\author[C. Jiang]{Chao Jiang }
\address{School of Mathematics and Physics, University of South China, Hengyang, P.R. China}
\email{1158911581@qq.com}
\author[FM. Cai]{Fangming Cai}
\address{School of Information Engineering, Jiangxi Polytechnic University, Jiujiang, P.R. China}
\email{caifengming2025@163.com}

\author[MQ. Yu]{MengQi Yu }
\address{School of Computational Science and Electronics, Hunan Institute of Engineering, Xiangtan, P.R. China}
\email{837915672@qq.com}


\date{\today}

\begin{document}
	
\newtheorem{theorem}{Theorem}[section]
\newtheorem{lemma}[theorem]{Lemma}
\newtheorem{proposition}[theorem]{Proposition}
\newtheorem{corollary}[theorem]{Corollary}
\theoremstyle{definition}
\newtheorem{definition}[theorem]{Definition}
\theoremstyle{remark}
\newtheorem{remark}{Remark}
\newtheorem{example}{Example}

\newcommand{\C}{{\mathbb C}}
\newcommand{\R}{{\mathbb R}}
\newcommand{\Z}{{\mathbb Z}}
\newcommand{\fH}{{\mathfrak H}}
\newcommand{\CH}{\mathbf H_{\C}}
\newcommand{\ip}[2]{\langle #1,\, #2\rangle}
\newcommand{\norm}[1]{\|#1\|}
\newcommand{\one}{\mathbf{1}}
\newcommand{\tr}{\mathrm{tr}}
\renewcommand{\Re}{\operatorname{Re}}
\renewcommand{\Im}{\operatorname{Im}}
\newcommand{\bignorm}[1]{\left|#1\right|}
\newcommand{\gauge}[1]{\left\|#1\right\|}

\begin{abstract}
For every $n\geqslant1$, the equilateral dimension of the Heisenberg group
$\fH^n=\C^n\times\R$ with the Kor\'anyi metric is $2n+2$.
\end{abstract}

	\maketitle
	
	\section{Introduction}\label{sec-intro}

	Equilateral sets are among the basic finite configurations in metric geometry.
	For a metric space $(X,d)$, the \emph{equilateral dimension}
	$\dim_{\mathrm E}(X)$ is the supremum of the cardinalities of finite subsets
	whose pairwise distances are equal.
	
	For the $n$-dimensional Euclidean space, the classical result is
	$\dim_{\mathrm E}(\mathbb R^n)=n+1$, with equality attained by the
	vertices of a regular simplex; see, for example,
	\cite{Blu}. On the unit sphere $S^{n-1}$, an equilateral set has at most $n+1$
	points, with equality attained by a regular simplex; the associated
	sharp spherical-cap packing is contained in Rankin's classical result
	\cite{Rnk}.
	
For every $n$-dimensional real normed space $X$, Petty and Soltan,
using the Danzer--Gr\"unbaum theorem on antipodal sets, proved that
$\dim_{\mathrm E}(X)\leq 2^n$, with equality if and only if $X$ is
linearly isometric to $\ell_\infty^n$ \cite{DG,P,So}. By contrast,
the exact value of $\dim_{\mathrm E}(\ell_1^n)$ remains unknown in
general. The $2n$ vertices of the cross-polytope form an equilateral
set, and Kusner conjectured that this configuration is maximal,
namely $\dim_{\mathrm E}(\ell_1^n)=2n$ \cite{Guy}. Equilateral sets
in $\ell_p^n$ have also been studied extensively; see, for example,
\cite{AP,Kus}. This normed-space theory does not apply directly to
the Heisenberg group, since the Kor\'anyi metric is not induced by a
norm on the underlying vector space.
	
	
	The first Heisenberg group $\fH^1=\C\times\R$ is equipped with its
	natural Kor\'anyi metric $d_{\mathrm{Cyg}}$, also known as the Cygan
	metric in complex hyperbolic geometry. Kim and Platis
	\cite{KP} proved that $\dim_{\mathrm E}(\fH^1)=4$,
	thereby giving a negative answer to a question of Chousionis--Tyson
	\cite{CT} on the existence of five pairwise equidistant points in
	$\fH^1$. The purpose of this paper is to determine the equilateral
	dimension of the Heisenberg groups $\fH^n=\C^n\times\R$ for every
	$n\geqslant1$.
	
	\begin{theorem}\label{thm-main}
		For every integer $n\geqslant 1$, the equilateral dimension of $\fH^n$
		satisfies
		\[
		\dim_{\mathrm E}(\fH^n)\;=\;2n+2.
		\]
		Equivalently, $\fH^n$ contains a $2n+2$-point equilateral set, and no
		equilateral set has more points.
	\end{theorem}
	\begin{remark}
		The value $2n+2$ also equals the Hausdorff dimension of $\fH^n$
		with respect to the Kor\'anyi metric.
	\end{remark}

The boundary of the Siegel domain $\partial\mathbf H_{\mathbb C}^{n+1}$
is the one-point compactification of $\fH^n$. Consider an equilateral
set of cardinality $m$ in $\fH^n$ and the standard lifts of its points
in the Hermitian space $V=\mathbb C^{n+1,1}$. Under these lifts,
pairwise equidistance of the points means exactly that the pairwise
Hermitian products of the lifts have a common modulus. The Gram matrix
of the lifts, however, is not determined by this condition. We pass
from each lift $v$ to the tensor $v\otimes\overline v$ in
$V\otimes\overline V$; the tensors of the lifts have an exact Gram
matrix. Adjoining the lift of the point at infinity increases the
order of this matrix by one, so that its negative inertia index is
$m$. Its negative inertia index is at most $2n+2$, the negative
inertia index of $V\otimes\overline V$, so we obtain the following
result.
	
	\begin{proposition}\label{prop-upper}
		Every equilateral subset of $\fH^n$ has cardinality at most $2n+2$.
	\end{proposition}
We prove the lower bound by constructing an explicit equilateral set
of $2n+2$ points in $\fH^n$. Following the reduction used in
\cite{KP}, two vertical points force the remaining points into the
horizontal subspace, where the distance conditions become constraints
on Hermitian inner products in $\C^n$; the construction is thereby
converted into a configuration problem of $2n$ vectors in $\C^n$.
These vectors are divided into two families of size $n$. Prescribing
the Gram matrix of the first family determines it up to the action of
$U(n)$ on $\mathbb C^n$; choosing a convenient representative, we
obtain the first family explicitly. Each vector of the second family
is determined by its inner products with the first family, and
requiring the second family to realize the same Gram matrix as the
first again reduces the problem to two scalar equations in two unit
complex numbers, which we solve explicitly. This yields the following
bound.
	
	\begin{proposition}\label{prop-lower}
		For every integer $n\geqslant1$, the Heisenberg group $\fH^n$
		contains an equilateral set of cardinality $2n+2$.
	\end{proposition}
	
	Theorem~\ref{thm-main} is an immediate consequence of
	Propositions~\ref{prop-upper} and~\ref{prop-lower}.
	
	The paper is organised as follows. In Section~\ref{sec-prelim} we recall
	the Heisenberg group, the Kor\'anyi metric, null lifts, and the
	Hermitian inertia theorem. Section~\ref{sec-upper} proves
	Proposition~\ref{prop-upper}, and Section~\ref{sec-lower} contains the
	reduction and the explicit construction establishing
	Proposition~\ref{prop-lower}.
	
	\section{Preliminaries}\label{sec-prelim}

The material in this section is standard. Further details may be found in
\cite{Gol}, except where otherwise indicated.

Throughout, all complex vector spaces are regarded as column vector spaces, and
$^{*}$ denotes the conjugate transpose. For
$\zeta,\eta\in\mathbb{C}^{n}$, let
\[
\langle \zeta,\eta\rangle
=
\eta^*\zeta
=
\sum_{j=1}^{n}\zeta_j\overline{\eta_j},
\qquad
\|\zeta\|^2
=
\langle\zeta,\zeta\rangle
\]
be the standard Hermitian inner product on $\mathbb{C}^{n}$.

Let $\mathbb{C}^{n+1,1}$ denote the $(n+2)$-dimensional complex vector space
equipped with the Hermitian form of signature $(n+1,1)$ given by
\[
\langle\mathbf{z},\mathbf{w}\rangle_H
=
\mathbf{w}^{*}H\mathbf{z}
=
\overline{w}_{n+1}z_0
+
\langle\zeta,\eta\rangle
+
\overline{w}_0z_{n+1},
\]
where
\[
\mathbf{z}
=
\begin{bmatrix}
	z_0\\
	\zeta\\
	z_{n+1}
\end{bmatrix},
\qquad
H=
\begin{bmatrix}
	0 & 0 & 1\\
	0 & I_n & 0\\
	1 & 0 & 0
\end{bmatrix},
\qquad
\mathbf{w}
=
\begin{bmatrix}
	w_0\\
	\eta\\
	w_{n+1}
\end{bmatrix}.
\]

The projectivisation of $\mathbb{C}^{n+1,1}$ is naturally identified with
$\mathbb{CP}^{n+1}$. A vector
$\mathbf{z}\in\mathbb{C}^{n+1,1}\setminus\{0\}$ is called
\emph{negative}, \emph{null}, or \emph{positive} according as
$\langle\mathbf{z},\mathbf{z}\rangle_H$ is negative, zero, or positive,
respectively. Complex hyperbolic $(n+1)$-space
$\mathbf{H}_{\mathbb{C}}^{n+1}$ is the set of negative vectors in
$\mathbb{CP}^{n+1}$, and its boundary
$\partial\mathbf{H}_{\mathbb{C}}^{n+1}$ consists of the null vectors.

By setting $z_{n+1}=1$, the Siegel domain together with its finite boundary is
given by
\[
\overline{\mathbf{H}_{\mathbb{C}}^{n+1}}\setminus\{q_\infty\}
=
\left\{
(z_0,\zeta)\in\mathbb{C}\times\mathbb{C}^{n}:
2\operatorname{Re}(z_0)+\|\zeta\|^2\leq 0
\right\}.
\]
The equality case gives the finite boundary
$\partial\mathbf{H}_{\mathbb{C}}^{n+1}\setminus\{q_\infty\}$.
For a finite boundary point $q=(z_0,\zeta)$, set
\[
z=\frac{\zeta}{\sqrt{2}}\in\mathbb{C}^{n},
\qquad
t=\operatorname{Im}(z_0).
\]
Then
\[
z_0=-\|z\|^2+it.
\]
We write $q=[z,t]$ and refer to $[z,t]$ as its
\emph{horospherical coordinates}. Its standard lift is
\[
\mathbf{q}
=
\begin{bmatrix}
	-\|z\|^2+it\\
	\sqrt{2}\,z\\
	1
\end{bmatrix}.
\]

The distinguished boundary point $q_\infty$ has standard lift
$$
\mathbf{q}_\infty=(1,\mathbf{0},0)^T,
$$
where $\mathbf{0}\in\mathbb{C}^{n}$ denotes the zero vector. It follows immediately from the above expressions that
$$
\langle \mathbf{q}_\infty,\mathbf{q}\rangle_H=1
$$
for every finite boundary point $q=[z,t]$.

The Bergman distance $\rho$ on
$\mathbf{H}_{\mathbb{C}}^{n+1}$ is given by
\[
\cosh^2\left(\frac{\rho(p,q)}{2}\right)
=
\frac{
	\langle\mathbf{p},\mathbf{q}\rangle_H
	\langle\mathbf{q},\mathbf{p}\rangle_H
}{
	\langle\mathbf{p},\mathbf{p}\rangle_H
	\langle\mathbf{q},\mathbf{q}\rangle_H
},
\]
where $\mathbf{p}$ and $\mathbf{q}$ are lifts of
$p,q\in\mathbf{H}_{\mathbb{C}}^{n+1}$, respectively.

For $[w,s]\in
\partial\mathbf{H}_{\mathbb{C}}^{n+1}\setminus\{q_\infty\}$,
consider the matrix
\[
T_{[w,s]}
=
\begin{bmatrix}
	1 & -\sqrt{2}\,w^* & -\|w\|^2+is\\
	0 & I_n & \sqrt{2}\,w\\
	0 & 0 & 1
\end{bmatrix}.
\]
A direct calculation shows that $T_{[w,s]}$ preserves
$\langle\cdot,\cdot\rangle_H$, and hence induces a Bergman isometry of
$\mathbf{H}_{\mathbb{C}}^{n+1}$. It fixes $q_\infty$ and acts on the
finite boundary by
\[
T_{[w,s]}[z,t]
=
\left[
w+z,\,
s+t+2\operatorname{Im}\langle w,z\rangle
\right].
\]
Such an isometry is called a \emph{Heisenberg translation}.

The composition of two Heisenberg translations satisfies
\[
T_{[z,t]}T_{[w,s]}
=
T_{\left[
	z+w,\,
	t+s+2\operatorname{Im}\langle z,w\rangle
	\right]}.
\]
This induces a group structure on
$\mathbb{C}^{n}\times\mathbb{R}$ given by
\[
[z,t]\cdot[w,s]
=
\left[
z+w,\,
t+s+2\operatorname{Im}\langle z,w\rangle
\right].
\]
The resulting group is the \emph{$n$-th Heisenberg group}, denoted by
\[
\mathfrak{H}^{n}
=
\mathbb{C}^{n}\times\mathbb{R}.
\]
The identity element is $[0,0]$, and $[z,t]^{-1}=[-z,-t].$
Thus the boundary of the Siegel domain may be regarded as the one-point
compactification of $\mathfrak{H}^{n}$.

The \emph{Cygan gauge} on $\mathfrak{H}^{n}$ is defined by
\[
\|[z,t]\|_{\mathrm{Cyg}}
=
\left|
\|z\|^2-it
\right|^{1/2}
=
\left(
\|z\|^4+t^2
\right)^{1/4}.
\]
The associated left-invariant metric is defined by
$$
d_{\mathrm{Cyg}}
\bigl([z,t],[w,s]\bigr)
=
\left\|
[w,s]^{-1}\cdot[z,t]
\right\|_{\mathrm{Cyg}}.
$$
With this normalization, it is commonly referred to as the Kor\'anyi metric in the Heisenberg-group literature and as the Cygan metric in complex hyperbolic geometry. Throughout the paper, we denote it by $d_{\mathrm{Cyg}}$.
Hence
$$
d_{\mathrm{Cyg}}
\bigl([z,t],[w,s]\bigr)
=
\left[
\|z-w\|^4
+
\left(
t-s+
2\operatorname{Im}\langle z,w\rangle
\right)^2
\right]^{1/4}.
$$

For boundary points
\[
p=[z,t],
\qquad
q=[w,s],
\]
a direct calculation gives
\[
\begin{aligned}
	\langle\mathbf{p},\mathbf{q}\rangle_H
	&=
	-\|z-w\|^2
	+i\left(
	t-s+2\operatorname{Im}\langle z,w\rangle
	\right).
\end{aligned}
\]
Therefore
\[
d_{\mathrm{Cyg}}(p,q)
=
\left|
\langle\mathbf{p},\mathbf{q}\rangle_H
\right|^{1/2}.
\]

For $\lambda>0$, the \emph{dilation}
\[
\delta_\lambda:\mathfrak{H}^{n}\longrightarrow\mathfrak{H}^{n},
\qquad
\delta_\lambda[z,t]=[\lambda z,\lambda^2t],
\]
is an automorphism of the Heisenberg group. Moreover,
\[
\|\delta_\lambda[z,t]\|_{\mathrm{Cyg}}
=
\lambda\|[z,t]\|_{\mathrm{Cyg}},
\]
and hence
\[
d_{\mathrm{Cyg}}
\bigl(\delta_\lambda p,\delta_\lambda q\bigr)
=
\lambda d_{\mathrm{Cyg}}(p,q)
\]
for all $p,q\in\mathfrak{H}^{n}$. In particular, any nonzero common
distance of an equilateral set may be normalized to $1$.

For \(U\in U(n)\), the standard Hermitian product on \(\mathbb C^n\) satisfies
$$
\langle Uz,Uw\rangle=\langle z,w\rangle.
$$
Consequently, the map
$$
R_U:\mathfrak{H}^{n}\longrightarrow\mathfrak{H}^{n},
\qquad
R_U[z,t]=[Uz,t],
$$
is a Cygan isometry.

We next recall Sylvester's law of inertia for Hermitian forms. For a
Hermitian matrix \(A=A^*\in M_m(\mathbb C)\), its \emph{inertia} is
$$
\operatorname{In}(A)
=
\bigl(\nu_+(A),\nu_-(A),\nu_0(A)\bigr),
$$
where \(\nu_+(A)\), \(\nu_-(A)\), and \(\nu_0(A)\) denote the numbers of
positive, negative, and zero eigenvalues of \(A\), counted with multiplicity.
For a Hermitian form \(\mathcal B\) on a finite-dimensional complex vector
space, we use the same notation for the inertia of a matrix representing
\(\mathcal B\).

\begin{theorem}[Sylvester's law of inertia, \cite{HJ}]\label{Syl}
	Let $A=A^*\in M_m(\mathbb{C})$ and
	$S\in\operatorname{GL}_m(\mathbb{C})$. Then
	\[
	\operatorname{In}(S^*AS)
	=
	\operatorname{In}(A).
	\]
	Equivalently, every Hermitian form admits a basis in which its matrix is
	\[
	\operatorname{diag}(I_p,-I_q,0_r),
	\]
	where the integers $p$, $q$, and $r$ are independent of the choice of
	basis.
\end{theorem}

	\section{The upper bound}\label{sec-upper}

We prove Proposition~\ref{prop-upper}. The proof uses the following two
standard algebraic tools: the
\emph{conjugate vector space} of $V=\C^{n+1,1}$ and the Hermitian form
induced on the tensor product $V\otimes_\C\overline V$.  A standard
reference for the conjugate space construction is
\cite{Petersen}.

The conjugate vector space $\overline{V}$ is the
vector space whose elements corresponding to $v\in V$ are denoted by
$\overline{v}$, with addition and scalar multiplication defined by
\[
\overline{u}+\overline{v}
=
\overline{u+v},
\qquad
\lambda\overline{v}
=
\overline{\overline{\lambda}v},
\]
for $u,v\in V$ and $\lambda\in\mathbb{C}$.

Define
\[
W=V\otimes_{\mathbb{C}}\overline{V}.
\]
For pure tensors, define
\[
\mathcal{H}
\bigl(
u\otimes\overline{v},
x\otimes\overline{y}
\bigr)
=
\langle u,x\rangle_H
\overline{\langle v,y\rangle_H},
\qquad
u,v,x,y\in V.
\]

To see that this definition is compatible with scalar multiplication, let
$\lambda\in\mathbb{C}$. Since
$\lambda\overline{v}=\overline{\overline{\lambda}v}$, we have
\[
\begin{aligned}
	\mathcal{H}
	\bigl(
	(\lambda u)\otimes\overline{v},
	x\otimes\overline{y}
	\bigr)
	&=
	\lambda\langle u,x\rangle_H
	\overline{\langle v,y\rangle_H} \\
	&=
	\mathcal{H}
	\bigl(
	u\otimes(\lambda\overline{v}),
	x\otimes\overline{y}
	\bigr).
\end{aligned}
\]
The same argument applies to the second variable, and compatibility with
addition follows directly from the additivity of
$\langle\cdot,\cdot\rangle_H$. Hence the above prescription extends uniquely
to a sesquilinear form $\mathcal{H}$ on $W$, which is Hermitian.

We next determine the inertia of $\mathcal{H}$.
\begin{lemma}\label{lem-tensor-signature}
The Hermitian form \(\mathcal{H}\) on
\(W=V\otimes_{\mathbb{C}}\overline{V}\) satisfies
\[\nu_+(\mathcal H)=(n+1)^2+1,
\qquad
\nu_-(\mathcal H)=2n+2.\]
\end{lemma}

\begin{proof}
	By Sylvester's law of inertia (Theorem~\ref{Syl}), there exists a basis
	$u_0,\ldots,u_{n+1}$ of $V$ such that
	\[
	\langle u_a,u_b\rangle_H
	=
	\varepsilon_a\delta_{ab},
	\]
	where $\delta_{ab}$ denotes the Kronecker delta and
	\[
	\varepsilon_0=\cdots=\varepsilon_n=1,
	\qquad
	\varepsilon_{n+1}=-1.
	\]
	The vectors
	\[
	u_a\otimes\overline{u_b},
	\qquad
	0\leq a,b\leq n+1,
	\]
	form a basis of $W$, and
	\[
	\mathcal{H}
	\bigl(
	u_a\otimes\overline{u_b},
	u_c\otimes\overline{u_d}
	\bigr)
	=
	\varepsilon_a\varepsilon_b
	\delta_{ac}\delta_{bd}.
	\]
	Hence
	\[
	\nu_+(\mathcal{H})=(n+1)^2+1,
	\qquad
	\nu_-(\mathcal{H})=2n+2.
	\]
\end{proof}

\begin{proof}[Proof of Proposition~\ref{prop-upper}]
	It suffices to consider a finite equilateral set
	\[ E=\{p_1,\ldots,p_m\}\subset\mathfrak{H}^{n}.\]
	By applying a dilation, we may assume that the common Cygan distance is $1$.
		
	For $1\leq i\leq m$, let $\xi_i$ be the standard lift of $p_i$, and set
	$\xi_0=\mathbf{q}_\infty$. According to Section~\ref{sec-prelim}, the $\xi_i$ are nonzero null vectors, and the Cygan distance is expressed in terms of the Hermitian product $\langle\cdot,\cdot\rangle_H$. Hence
	\begin{equation}\label{eq-null-lift-relations}
		\langle \xi_i,\xi_i\rangle_H=0,
		\qquad
		\left|
		\langle \xi_i,\xi_j\rangle_H
		\right|
		=
		1
		\quad
		(0\leq i\neq j\leq m).
	\end{equation}
	
	The relations \eqref{eq-null-lift-relations} determine only the moduli of the Hermitian products
	\(\langle\xi_i,\xi_j\rangle_H\), so the Gram matrix of the lifts is not determined by the equilateral condition. Therefore, we pass to \(V\otimes_{\mathbb C}\overline V\) to remove this phase ambiguity.
	 
	For $0\leq i\leq m,$ let $X_i=\xi_i\otimes\overline{\xi_i}.$ Then
	\[
	\mathcal{H}(X_i,X_j)
	=
	\langle \xi_i,\xi_j\rangle_H
	\overline{\langle \xi_i,\xi_j\rangle_H}
	=
	\left|\langle \xi_i,\xi_j\rangle_H\right|^2.
	\]
	By the relations \eqref{eq-null-lift-relations}, the Gram matrix of
	$X_0,\ldots,X_m$ is
	\[
	C=J_{m+1}-I_{m+1},
	\]
	where all entries of $J_{m+1}$ are equal to 1.
	
	Let $\mathbf{1}=(1,\ldots,1)^{\top}\in\mathbb{C}^{m+1}$ and
	$y\in\mathbf{1}^{\perp}$. It is easy to see that
	\[
	C\mathbf{1}=m\mathbf{1},
	\qquad
	Cy=-y.
	\]
	Since $\dim\mathbf{1}^{\perp}=m$, it follows that $\nu_-(C)=m$.
	
It remains to show that $\nu_-(C)\leq\nu_-(\mathcal{H})$.
Define
$$
T:\mathbb{C}^{m+1}\longrightarrow W,
\qquad
T(y_0,\ldots,y_m)
=
\sum_{i=0}^{m}y_iX_i.
$$
For $y\in\mathbf{1}^{\perp}$, we have
\begin{align*}
	\mathcal{H}(Ty,Ty)
	&=
	\sum_{i,j=0}^{m}
	y_i\overline{y_j}\mathcal{H}(X_i,X_j)\\
	&=
	y^*Cy
	=
	-\|y\|^2.
\end{align*}
Hence $T$ is injective on $\mathbf{1}^{\perp}$, and
$T(\mathbf{1}^{\perp})$ is an $m$-dimensional negative definite subspace of $W$. Therefore
$$
m
=
\nu_-(C)
\leq
\nu_-(\mathcal{H})
=
2n+2.
$$
This proves the proposition.
\end{proof}

\section{The lower bound}\label{sec-lower}

We prove Proposition~\ref{prop-lower} by an explicit construction. The construction is organized as follows.
Lemma~\ref{lemma-lower-reduction} transforms the construction of the
equilateral set into a configuration problem of \(2n\) vectors in the
\(n\)-dimensional complex vector space \(\mathbb C^n\).
Lemma~\ref{lem-first-family} gives the explicit construction of the
first \(n\) vectors. Lemma~\ref{lem-compatible-vector} characterizes
the vectors compatible with the first family, and
Lemma~\ref{lem-second-family} gives the construction of the remaining
\(n\) vectors. The proof of Proposition~\ref{prop-lower} at the end of
the section combines these ingredients into the required
\((2n+2)\)-point equilateral set, treating the case \(n=1\) separately.
\begin{lemma}\label{lemma-lower-reduction}
	Let
	\[
	p_{+}=[\mathbf 0,\tfrac12],\qquad
	p_{-}=[\mathbf 0,-\tfrac12],\qquad
	p_i=[\mathbf z_i,0],\quad i=1,\ldots,2n.
	\]
	Then \(p_{+},p_{-},p_1,\ldots,p_{2n}\) form an equilateral set
	with common Cygan distance \(1\) if and only if
	\begin{align}
		\|\mathbf z_i\|^2&=\frac{\sqrt3}{2},
		&&i=1,\ldots,2n, \label{eqA}\\
		\left|\ip{\mathbf z_i}{\mathbf z_j}-\frac{\sqrt3}{2}\right|^2&=\frac14,
		&&1\leq i<j\leq 2n. \label{eqB}
	\end{align}
\end{lemma}
\begin{proof}
	The bisector of \(p_{+}\) and \(p_{-}\) is the horizontal subspace
	\(\C^n\times\{0\}\), since for \(p=[\mathbf z,t]\),
	\[
	d_{\mathrm{Cyg}}(p,p_\pm)^4=\|\mathbf z\|^4+\left(t\mp\frac12\right)^2,
	\]
	and these two quantities are equal if and only if \(t=0\).
	
	It is easy to see that $d_{\mathrm{Cyg}}(p_+,p_-)^4=1$. For \(p_i=[\mathbf z_i,0]\), we have
	\[
	d_{\mathrm{Cyg}}(p_i,p_{\pm})^4
	=
	\|\mathbf z_i\|^4+\frac14
	=
	1
	\quad\Longleftrightarrow\quad
	\|\mathbf z_i\|^2=\frac{\sqrt3}{2}.
	\]
	Using  \(\|\mathbf z_i\|^2=\|\mathbf z_j\|^2=\frac{\sqrt3}{2}\), we have
	\[
	d_{\mathrm{Cyg}}(p_i,p_j)^4
	=
	4\left|\ip{\mathbf z_i}{\mathbf z_j}-\frac{\sqrt3}{2}\right|^2.
	\]
	Hence
	\[
	d_{\mathrm{Cyg}}(p_i,p_j)=1
	\quad\Longleftrightarrow\quad
	\left|
	\ip{\mathbf z_i}{\mathbf z_j}
	-
	\frac{\sqrt3}{2}
	\right|^2
	=
	\frac14.
	\]
	This proves the equivalence.
\end{proof}

For the remainder of the construction, we assume that \(n\geqslant2\). The case \(n=1\) will be treated separately in the proof of Proposition~\ref{prop-lower}.

According to Section~\ref{sec-prelim}, the action of \(U(n)\) on \(\mathbb C^n\) preserves the Hermitian product and induces Cygan isometries of \(\mathfrak H^n\). In particular, the conditions \eqref{eqA} and \eqref{eqB} are \(U(n)\)-invariant. This suggests looking for the first \(n\) vectors in a convenient triangular form. The following lemma gives an explicit construction.

\begin{lemma}\label{lem-first-family}
	For \(1\leqslant k\leqslant n\), define
	\[
	\mathbf z_k
	=
	(d_1,\ldots,d_{k-1},s_k,0,\ldots,0)^T
	\in\mathbb R^n,
	\]
	where
	\[
	d_k=\sqrt{\frac1{2(k+\alpha)(k+\alpha-1)}},
	\qquad
	s_k=\sqrt{\frac{k+\alpha}{2(k+\alpha-1)}}, \qquad \alpha=\frac{\sqrt3+1}{2}.
	\]
	Then $\mathbf z_1,\ldots,\mathbf z_n$ satisfy equation \eqref{eqA}, and satisfy
	equation \eqref{eqB} for $1\leqslant i<j\leqslant n$.
\end{lemma}
\begin{proof}
	Notice that
	\[
	d_i^2
	=
	\frac12
	\left(
	\frac1{i+\alpha-1}
	-
	\frac1{i+\alpha}
	\right),
	\qquad
	d_is_i
	=
	\frac1{2(i+\alpha-1)}.
	\]
	A direct computation gives
	\begin{align*}
		\|\mathbf z_k\|^2
		&=
		\sum_{i=1}^{k-1}d_i^2+s_k^2
		=
		\frac1{2\alpha}+\frac12
		=
		\frac{\sqrt3}{2},\\
		\langle\mathbf z_j,\mathbf z_k\rangle
		&=
		\sum_{i=1}^{j-1}d_i^2+d_js_j
		=
		\frac1{2\alpha}
		=
		\frac{\sqrt3-1}{2},
		\qquad j<k.
	\end{align*}
	Therefore, \(\mathbf z_1,\ldots,\mathbf z_n\) satisfy equation \eqref{eqA}, and equation \eqref{eqB} holds for \(1\leqslant j<k\leqslant n\).
\end{proof}
\begin{remark}
	Let
	\[
	\mathsf Z_1
	=
	[\mathbf z_1\mid\cdots\mid\mathbf z_n].
	\]
	By Lemma~\ref{lem-first-family}, its Gram matrix is
	\[
	\mathsf G=\mathsf Z_1^T\mathsf Z_1=\frac12 I_n+\frac{\sqrt3-1}{2}J_n,
	\]
	where \(J_n\) denotes the \(n\times n\) all-ones matrix.
	The matrix \(\mathsf G\) is positive definite. By the uniqueness of the Cholesky factorization, the upper-triangular factor with positive diagonal entries is precisely \(\mathsf Z_1\).
\end{remark}

To construct the remaining \(n\) vectors, we first characterize those vectors that satisfy equation \eqref{eqB} with \(\mathbf z_1,\ldots,\mathbf z_n\).
\begin{lemma}\label{lem-compatible-vector}
	A vector \(\mathbf z\in\mathbb C^n\) satisfies equation \eqref{eqB} with
	\(\mathbf z_1,\ldots,\mathbf z_n\) if and only if
	\[
	\mathbf z
	=
	\mathsf Z_1^{-T}\mathbf b,
	\]
	where
	\[
	\mathbf b
	=
	(b_1,\ldots,b_n)^T,
	\qquad
	b_i
	=
	\frac{\sqrt3}{2}+\frac12u_i,
	\qquad
	|u_i|=1,
	\qquad
	i=1,\ldots,n.
	\]
\end{lemma}
\begin{proof}
	Equation \eqref{eqB} is equivalent to the existence of unit complex numbers
	\(u_1,\ldots,u_n\) such that
	\[
	\langle\mathbf z,\mathbf z_i\rangle
	=
	\frac{\sqrt3}{2}+\frac12u_i
	=
	b_i,
	\qquad
	i=1,\ldots,n.
	\]
	Since \(\mathbf z_1,\ldots,\mathbf z_n\) are real, these relations are equivalent to
	\[
	\mathsf Z_1^T\mathbf z=\mathbf b.
	\]
	As \(\mathsf Z_1\) is invertible, this is equivalent to
	\[
	\mathbf z=\mathsf Z_1^{-T}\mathbf b.
	\]
\end{proof}
We now use Lemma~\ref{lem-compatible-vector} to construct the remaining \(n\) vectors. We choose the corresponding phase vectors in a symmetric form involving only two unit complex numbers.
\begin{lemma}\label{lem-second-family}
	Let
	\begin{align}\label{unit}
		v=\frac{
			-\left(n-1+\frac{1}{2\sqrt3}\right)
			+i\sqrt{\,n-\frac{n-1}{\sqrt3}-\frac{1}{12}\,}
		}{
			n-1+e^{i\pi/3}
		},
		\qquad
		u
		=
		ve^{i\pi/3}.
	\end{align}
	Then $u$ and $v$ are unit complex numbers.
	
	For \(1\leqslant i\leqslant n\), define
	\begin{align}\label{the-second-family}
		\mathbf z_{n+i}	=\mathsf Z_1^{-T}\mathbf b^{(i)},
	\end{align}
	where
	\[
	b^{(i)}_j
	=
	\begin{cases}
		\frac12(\sqrt3+u), & j=i,\\[2ex]
		\frac12(\sqrt3+v), & j\neq i.
	\end{cases}
	\]
	Then \(\mathbf z_{n+1},\ldots,\mathbf z_{2n}\) satisfy
	equation \eqref{eqA}, and satisfy equation \eqref{eqB} for
	\(n+1\leqslant i<j\leqslant2n\).
\end{lemma}
\begin{proof}
	Let
	\[
	\mathsf B
	=
	[\mathbf b^{(1)}\mid\cdots\mid\mathbf b^{(n)}],
	\qquad
	\mathsf Z_2
	=
	[\mathbf z_{n+1}\mid\cdots\mid\mathbf z_{2n}].
	\]
	By the definition of \(\mathbf b^{(i)}\), we have
	\[
	\mathsf B
	=
	\frac{u-v}{2}I_n
	+
	\frac{\sqrt3+v}{2}J_n\qquad\text{ and }\qquad\mathsf Z_2
	=
	\mathsf Z_1^{-T}\mathsf B.
	\]
	To verify equations \eqref{eqA} and \eqref{eqB} for the second family, it is enough to show that
	\[
	\|\mathbf z_{n+i}\|^2
	=
	\frac{\sqrt3}{2},
	\qquad
	\langle\mathbf z_{n+i},\mathbf z_{n+j}\rangle
	=
	\frac{\sqrt3-1}{2},
	\qquad
	1\leqslant i<j\leqslant n.
	\]
	Equivalently, its Gram matrix is \(\mathsf G\), that is,
\[
\mathsf Z_2^{*}\mathsf Z_2=\mathsf B^{*}\mathsf Z_1^{-1}\mathsf Z_1^{-T}\mathsf B
=\mathsf B^{*}\mathsf G^{-1}\mathsf B=\mathsf G.
\]

Let
\[P=I_n-\frac1nJ_n,\qquad Q=\frac1nJ_n.\]
It is easy to see that
\[P^2=P,\qquad Q^2=Q,\qquad PQ=QP=0,\qquad P+Q=I_n.\]
Therefore, \(\mathsf G\), \(\mathsf B\), and \(\mathsf G^{-1}\) admit the decompositions
\begin{align*}
	\mathsf G
	&=
	\frac12P
	+
	\frac{n\sqrt3-n+1}{2}Q,\\
	\mathsf B
	&=
	\frac{u-v}{2}P
	+
	\frac{n\sqrt3+u+(n-1)v}{2}Q,\\
	\mathsf G^{-1}
	=
	2I_n&
	-
	\frac{2(\sqrt3-1)}{n\sqrt3-n+1}J_n=
	2P
	+
	\frac{2}{n\sqrt3-n+1}Q.
\end{align*}
Substituting these decompositions into
\(\mathsf B^{*}\mathsf G^{-1}\mathsf B=\mathsf G\), we obtain
\[
\frac{|u-v|^2}{2}P
+
\frac{
	|n\sqrt3+u+(n-1)v|^2
}{
	2(n\sqrt3-n+1)
}Q
=
\frac12P
+
\frac{n\sqrt3-n+1}{2}Q.
\]
Since \(P\) and \(Q\) are linearly independent, this is equivalent to
\begin{align}
	|u-v| &= 1, \label{eq-P}\\
	\left|n\sqrt3+u+(n-1)v\right| &= n\sqrt3-n+1. \label{eq-Q}
\end{align}

It remains to verify equations \eqref{eq-P} and \eqref{eq-Q}. By the definition of
\(u\) and \(v\) in equation \eqref{unit},
	\[
	v
	=
	\frac{
		-\left(n-1+\frac{1}{2\sqrt3}\right)
		+
		i\sqrt{
			n-\frac{n-1}{\sqrt3}-\frac{1}{12}
		}
	}{
		n-1+e^{i\pi/3}
	},
	\qquad
	u
	=
	ve^{i\pi/3}.
	\]
For \(n\geqslant2\), it is easy to check that
\[
\begin{gathered}
	n-\frac{n-1}{\sqrt3}-\frac{1}{12}>0,\\[1ex]
	\left|n-1+e^{i\pi/3}\right|^2
	=
	n^2-n+1,\\[1ex]
	\left|
	-\left(n-1+\frac{1}{2\sqrt3}\right)
	+
	i\sqrt{
		n-\frac{n-1}{\sqrt3}-\frac{1}{12}
	}
	\right|^2
	=
	n^2-n+1.
\end{gathered}
\]
Hence \(|u|=|v|=1\), and
\(|u-v|=|v|\,|e^{i\pi/3}-1|=1\), so equation \eqref{eq-P} holds.
	
	On the other hand,
	\[
	\begin{aligned}
		u+(n-1)v
		&=
		v\left(e^{i\pi/3}+n-1\right)\\
		&=
		-\left(n-1+\frac{1}{2\sqrt3}\right)
		+
		i\sqrt{
			n-\frac{n-1}{\sqrt3}-\frac{1}{12}
		}.
	\end{aligned}
	\]
Hence
\[
\begin{aligned}
	|n\sqrt3+u+(n-1)v|^2
	&=
	\left(
	n\sqrt3-n+1-\frac{1}{2\sqrt3}
	\right)^2
	+
	n-\frac{n-1}{\sqrt3}-\frac{1}{12}\\
	&=
	(n\sqrt3-n+1)^2.
\end{aligned}
\]
Therefore equation \eqref{eq-Q} holds.
\end{proof}

\begin{proof}[Proof of Proposition~\ref{prop-lower}]
	\textbf{The case \(n=1\).} Consider the four points
	\[
	p_+=[0,\tfrac12],
	\qquad
	p_-=[0,-\tfrac12],
	\qquad
	p_1=[r,0],
	\qquad
	p_2=[re^{i\theta},0],
	\]
	where \(r=(3/4)^{1/4}\) and \(\cos\theta=\frac56\). It is easy to
	see that they are pairwise at Cygan distance \(1\), so \(\fH^1\)
	contains an equilateral set of cardinality \(4=2n+2\).
	
	\textbf{The case \(n\geqslant2\).} The construction consists of the
	following three layers:
	\begin{enumerate}
		\item the two vertical points
		\(p_\pm=[\mathbf 0,\pm\tfrac12]\);
		\item the vectors \(\mathbf z_1,\ldots,\mathbf z_n\) given by
		Lemma~\ref{lem-first-family};
		\item the vectors \(\mathbf z_{n+1},\ldots,\mathbf z_{2n}\) given
		by Lemma~\ref{lem-second-family}.
	\end{enumerate}
	
	With the first layer fixed, Lemma~\ref{lemma-lower-reduction} reduces
	the task to verifying equation \eqref{eqA} for each vector and
	equation \eqref{eqB} for each pair of vectors. By
	Lemma~\ref{lem-first-family}, the first family satisfies
	equations \eqref{eqA} and \eqref{eqB}. Since \(u\) and \(v\) are
	unit complex numbers, every component of each \(\mathbf b^{(i)}\) is
	of the form \(\frac{\sqrt3}{2}+\frac12\omega\) with
	\(|\omega|=1\); by Lemma~\ref{lem-compatible-vector}, each
	\(\mathbf z_{n+i}\) therefore satisfies equation \eqref{eqB} with
	each of \(\mathbf z_1,\ldots,\mathbf z_n\). By
	Lemma~\ref{lem-second-family}, the second family likewise satisfies
	equations \eqref{eqA} and \eqref{eqB}.
	
	Thus the hypotheses of Lemma~\ref{lemma-lower-reduction} hold, and
	the points
	\[
	\left\{p_+,\,p_-,\,p_1,\ldots,p_{2n}\right\}
	\]
	form an equilateral set of cardinality \(2n+2\).
\end{proof}

%
%

\end{document}